\documentclass[11pt,a4paper]{amsart}
\usepackage{amssymb,amsmath}
\numberwithin{equation}{section}
\vspace{8cm}
\newtheorem{theorem}{Theorem}[section]

\newtheorem{corollary}[theorem]{Corollary}
\newtheorem{lemma}[theorem]{Lemma}

\newtheorem{remark}[theorem]{Remark}

\newcommand{\R}{{\mathbb R}}
\newcommand{\Rn}{{\mathbb R}^n}
\newcommand{\Semi}{{\mathbb R}^n_+}

\newcommand{\ro}{\rho}

\newcommand{\Ll}{\frac{\Lambda}{\lambda}}
\newcommand{\lL}{\frac{\lambda}{\Lambda}}

\newcommand{\beq}{\begin{equation}}
\newcommand{\eeq}{\end{equation}}
\newcommand{\reff}[1]{{\rm (\ref{#1})}}
\newcommand{\dis}{\displaystyle}

\newcommand{\noi}{\noindent}

\newcommand{\Mm}{\mathcal{M}^-_{\lambda ,\Lambda}}
\newcommand{\Mp}{\mathcal{M}^+_{\lambda ,\Lambda}}

\begin{document}

\title [Fully nonlinear elliptic inequalities in halfspaces]{Explicit  subsolutions and a  Liouville theorem  for fully nonlinear uniformly elliptic inequalities in halfspaces}

\author[F. Leoni]{Fabiana  Leoni}

\address{Sapienza  Universit\`a di Roma \\
P.le Aldo Moro 2, 00185 Roma, Italy\\
leoni@mat.uniroma1.it}

\date{December 5, 2011}



\begin{abstract}
We prove a Liouville type theorem for arbitrarily growing positive  viscosity supersolutions of fully nonlinear uniformly elliptic  equations in halfspaces. Precisely, let $\Mm$ be the Pucci's inf--operator, defined as the infimum of all linear uniformly elliptic operators with ellipticity constants $\Lambda \geq \lambda >0$. Then, we prove that  the inequality $\Mm (D^2u) +u^p\leq 0$ does not have any positive viscosity solution in a halfspace provided that  $-1\leq p \leq \frac{\Ll n+1}{\Ll n-1}$, whereas positive solutions do exist if either $p< -1$ or $p> \frac{\Ll (n-1)+2}{\Ll (n-1)}$. This will be accomplished by constructing explicit subsolutions of the homogeneous equation $\Mm (D^2u)=0$  and by proving a nonlinear version in a halfspace of the classical Hadamard three-circles theorem for entire superharmonic functions.
\end{abstract}

\maketitle

\section{Introduction}

We focus on  positive supersolutions of second order  fully nonlinear uniformly elliptic equations of the form either
\beq \label{F}
F(x, D^2u) =0 \qquad \hbox{ in }\ \Semi\
\eeq
or
\beq \label{Fp}
F(x, D^2u) + u^p=0 \qquad \hbox{ in }\ \Semi\, ,
\eeq
where  $\Semi$ is the halfspace $\{ x=(x', x_n)\in \R^{n-1}\times \R \ :\ x_n>0\}$, with $n\geq 2$. Here $F: \Semi \times \mathcal{S}_n \to \R$ is a continuous function of the space variable $x\in \Semi$ and of the Hessian matrix $D^2u\in  \mathcal{S}_n$, the set of symmetric $n\times n$ matrices. 

For equation \reff{F} we first construct some explicit homogeneous subsolutions, vanishing on the boundary $\partial \Semi \setminus \{0\}$, and then we use them to derive lower bounds and monotonicity properties for  nonnegative  supersolutions. The result we obtain closely resembles the classical Hadamard three--spheres theorem for bounded from below superharmonic functions,  and   it will be applied in order to obtain a Liouville type theorem  for positive supersolutions of \reff{Fp}.

Let us recall that the Liouville property for equations posed in halfspaces and having power--like zero order terms is one of the crucial steps for applying the blow--up method developed in \cite{GS}, which yields
 $L^\infty$ a priori estimates  for solutions of boundary value problems in bounded domains. Liouville type properties  have been largely studied mainly in case of semilinear equations, and our contribution is devoted to the extension to the fully nonlinear framework.

We assume that the operator $F$ is uniformly elliptic with ellipticity constants $\Lambda \geq \lambda >0$, that is $F$ is assumed to satisfy
\beq \label{eu}
\lambda\, {\rm tr}P \leq F(x, M+P)- F(x,M)\leq \Lambda \, {\rm tr}P
\eeq
for all $x\in \Semi$ and for every $M,\ P \in \mathcal{S}_n$, with $P\geq O$ (i.e. nonnegative definite).

We further assume that $F(x, O)=0$, so that inequalities \reff{eu} amount to
$$
\lambda \, {\rm tr} M^+ - \Lambda \, {\rm tr} M^- \leq F(x, M)\leq \Lambda \, {\rm tr} M^+ -\lambda \, {\rm tr} M^-
$$
for all $x\in \Semi$ and $M\in \mathcal{S}_n$, where $M^+,\ M^- \geq O$ are the only nonnegative definite matrices decomposing $M$ as $M=M^+-M^-$ and satisfying $M^+ M^-=0$. Let us recall that the left and the right hand side of the above inequality represent the Pucci extremal operators (see e.g. \cite{cafca}), that are the special uniformly elliptic operators given by
$$
\begin{array}{c}
\dis \Mm (M)= \lambda \, \sum_{\mu_i >0} \mu_i +\Lambda \, \sum_{\mu_i <0} \mu_i =\inf_{A\in \mathcal{A}_{\lambda ,\Lambda}} {\rm tr}( AM)\\[2ex]
\dis  \mathcal{M}^+_{\lambda ,\Lambda}  (M)= \Lambda \, \sum_{\mu_i >0} \mu_i + \lambda \, \sum_{\mu_i <0} \mu_i =\sup_{A\in \mathcal{A}_{\lambda ,\Lambda}} {\rm tr}( AM)
\end{array}
$$
where $\mu_1 ,\ldots , \mu_n$ stand for the eigenvalues of $M$ and $ \mathcal{A}_{\lambda ,\Lambda}$ is the set of all symmetric matrices whose eigenvalues belong to the closed interval $[ \lambda , \ \Lambda ]$. Thus, the uniform ellipticity condition \reff{eu} is equivalent for the operator $F$ to satisfy
$$
\Mm (M) \leq F(x,M) \leq  \mathcal{M}^+_{\lambda ,\Lambda}  (M)
$$
for every $x$ and every $M$, and  this implies  that  if $u$ is a solution (or a supersolution) either of  \reff{F} or of \reff{Fp}, then $u$ satisfies respectively either
\beq \label{mm}
\Mm (D^2u) \leq 0 \quad \hbox{ in }\ \Semi 
\eeq
or
\beq \label{mmeq}
\Mm (D^2u) +u^p\leq 0 \quad \hbox{ in }\ \Semi\, .
\eeq
In this respect, \reff{mm} and \reff{mmeq} are  the inequalities  naturally associated with all uniformly elliptic equations of the form either \reff{F} or \reff{Fp} respectively.

Our goal is to identify an explicit  range of values for the exponent $p$ for which \reff{mmeq} does not admit positive solutions. Note that weak solutions of inequality \reff{mmeq}, because of   non divergence form of the principal part, have to be meant in the viscosity sense, and we refer to \cite{cafca, user} for the  viscosity solutions theory for Pucci and more general fully nonlinear operators.

As a consequence of our results, we obtain  the following theorem.

\begin{theorem} \label{Lio} Let $n\geq 2$ and $-1\leq p\leq \frac{\Ll n+1}{\Ll n-1}$. Then, there does not exist any positive viscosity solution of inequality  \reff{mmeq}.
\end{theorem}

If $\Lambda =\lambda$, then \reff{mmeq} becomes, up to a scaling factor for  the function $u$, the semilinear inequality
\beq \label{lin}
\Delta u +u^p\leq 0\, ,
\eeq
and Theorem \ref{Lio} thus gives an  extension of the well known fact that  inequality \reff{lin} does not have positive solutions in a halfspace for $-1\leq p\leq \frac{n+1}{n-1}$ (see e.g. \cite{AS}).  In other words, $-1$ and  $\frac{\Ll n+1}{\Ll n-1}$ work as  critical exponents for the Liouville property for operator $\Mm$ in a halfspace.

To show the existence of critical exponents for inequality \reff{mmeq} we can apply  the same argument used in \cite{KLS} for linear equations. Indeed, a straightforward computation shows that if $p>1$ (or if $p<1$) and $u$ is a positive  solution of inequality \reff{mmeq}, then for any $q>p$ (or $q<p$, respectively) the function $v=\left( \frac{p-1}{q-1}\right)^{1/(q-1)} u^{(p-1) / (q-1)}$ satisfies
$$
\Mm (D^2v) +v^q\leq 0 \quad \hbox{ in }\ \Semi\, .
$$
Therefore one can define the exponents
$$
\begin{array}{c}
p^* = \inf \{ p>1\, :\, \hbox{ \reff{mmeq} has a positive  solution}\}\\[2ex]
p_* = \sup \{ p<1\, :\, \hbox{ \reff{mmeq} has a positive  solution}\}
\end{array}
$$
 and the Liouville property for inequality \reff{mmeq} certainly fails if either $p<p_*$ or $p>p^*$.

At this point let us recall that inequalities such as \reff{lin} and \reff{mmeq} have  been extensively treated and subjected to different generalizations in past and recent works. For linear operators and  inequalities posed  in the whole space or in exterior domains, we just mention  \cite{G} for supersolutions of \reff{lin},  \cite{KLS} for uniformly elliptic non constant coefficient inequalities  of the form 
\beq \label{linA}
{\rm tr} (A(x)D^2u) +u^p\leq 0
\eeq
 and \cite{BCDC} for inequalities involving the Heisenberg--Laplace operator.
 
  In the  fully nonlinear case, inequalities posed in the whole space or in exterior domains have been considered  for Pucci extremal operators in  \cite{CL} for $p\geq 0$ and in \cite{AS1} for $p<0$, in  \cite{CDC} for Pucci extremal operators plus  first order terms, in  \cite{AS, BD, FQ1} for more general classes of fully nonlinear operators and zero order terms, and in   \cite{FQ} for fully nonlinear integrodifferential operators.  We merely recall that when inequality \reff{mmeq} is considered in the whole space, then the critical exponents are 
  $$p_*=\left\{
  \begin{array}{ll}
  \dis -\infty & \hbox{if }\ \dis \lL(n-1)\geq 1\\
  \dis \frac{\lL (n-1)+1}{\lL (n-1)-1} &   \hbox{if }\ \dis \lL(n-1)< 1
  \end{array}\right.
  \quad \hbox{and }\quad p^*=\frac{\Ll (n-1)+1}{\Ll (n-1)-1}
  $$
   and the Liouville property holds if and only if   $p_*\leq p\leq p^*$. 
 \smallskip
 
Inequality \reff{lin} posed in an halfspace or in more general cone--like domains has  been studied in \cite{BCDN,  KLM, KLMS}, and recently revised in \cite{AS}. 
In particular, the arguments  used in \cite{AS} can be applied also to fully nonlinear principal parts, and this is, up to our knowledge, the only existing result for non divergence form differential inequalities posed in conical domains, including the linear case of \reff{linA}.

The results of \cite{AS} in particular relate 
 the critical exponents $p^*,\ p_*$ for \reff{mmeq} to the scaling exponents $\alpha^\pm$ of the homogeneous solutions of the homogeneous equation. 
Precisely, we recall that, in view of the results of \cite{Miller} and their recent extensions in \cite{ASS},  the extremal homogeneous equation
\beq \label{omo}
\Mm (D^2 \Phi)=0
\eeq
is known to have in any cone $\mathcal{C}_\sigma =\{ x\in \R^n\, :\ x_n> \sigma \, |x|\}$, with $-1<\sigma <1$, exactly two solutions, up to normalization,  of  the form
$$
\Phi_{\alpha_\sigma^{\pm}} (x)= |x|^{-\alpha_\sigma^{\pm}} \phi_{\alpha_\sigma^{\pm}}\left(  \frac{x_n}{|x|} \right)
$$
with $\alpha_\sigma^- <0 <\alpha_\sigma^+$ and $\phi_{\alpha_\sigma^{\pm}}$  $C^2$--functions defined on the interval $[\sigma , 1 ]$  satisfying $\phi_{\alpha_\sigma^{\pm}} ( \sigma )=0$ and $\phi_{\alpha_\sigma^{\pm}}(t )>0$ for $ \sigma <t\leq 1$.

By applying the proof of Theorem 5.1 in \cite{AS} with the functions $\Psi^{\pm}$ there replaced by $\Phi_{\alpha_\sigma^{\pm}}$, it follows that positive supersolutions in $\mathcal{C}_\sigma$ of \reff{mmeq} do not exist if and only if
$$
1+\frac{2}{\alpha_\sigma^-} \leq p \leq 1+\frac{2}{\alpha_\sigma^+}\, .
$$
Now, for the halfspace $\Semi = \mathcal{C}_0$ it is clear that $\alpha_0^-=-1$ (and $\phi_{\alpha_\sigma^-} (t) =t$). Therefore, in this case, if we set $\alpha_0^+=\alpha$,  then we have
\beq \label{p*}
p_*= -1 \hbox{ and }\   p^* =1+\frac{2}{\alpha}\, .
\eeq
On the other hand, the existence of the homogeneous solution $\Phi_{\alpha}$ is obtained in \cite{Miller} by means of an abstract existence and uniqueness result for nonlinear ODEs having singular monotone lower order terms, and in \cite{ASS} by using a topological argument which leads to a fixed point theorem in Banach spaces. In both cases, the exponent $\alpha$ is not or not sharply estimated from above, so that no specific lower bound for $p^*$ can be deduced.

By the comparison principles of Phragm\'en--Lindel\"of type given in \cite{ASS, Miller}, $\alpha$ can be estimated from above provided that an explicit  subsolution of \reff{omo} vanishing on $\partial \Semi \setminus \{0\}$ is known, as well as an homogeneous  supersolution of \reff{omo} vanishing for  $|x|\to \infty$ produces a lower bound for $\alpha$. In \cite{Miller}, only a supersolution
of \reff{omo} is exhibited, namely the function
$$
\hat \Phi = \frac{x_n}{|x|^{\Ll (n-1)+1}}\, .
$$
Note that the inequality $\Mm (D^2 \hat \Phi )\leq 0$ in $\Semi$ easily follows from the fact that $\Mm$ is superadditive and $\hat \Phi$ is, up to a negative constant, the partial derivative with respect to $x_n$ of a well known radial solution for $\Mm$ in $\R^n\setminus \{ 0\}$. The homogeneous supersolution $\hat \Phi$ gives the lower bound $\alpha \geq \Ll (n-1)$, wich in turn implies, by \reff{p*},
$$
p^*\leq \frac{ \Ll (n-1)+2}{\Ll (n-1)}\, .
$$
In other words, inequality \reff{mmeq} does admit  positive solutions for $p>\frac{ \Ll (n-1)+2}{\Ll (n-1)}$, $u(x)=\frac{x_n}{|x|^\beta}$ being an explicit supersolution of \reff{mmeq}  for $\frac{p+1}{p-1}< \beta < \Ll (n-1)+1$.

Therefore, in order to obtain a nonexistence statement as in Theorem \ref{Lio}, we have to determine an explicit  subsolution of \reff{omo} vanishing on $\partial \Semi \setminus \{0\}$.
This turns out to be a non trivial task, since   the standard separation of variables technique in polar representation  hardly applies to operator $\Mm$ . To appreciate the strongly nonlinear character of $\Mm$, note that, for $n=2$, equation \reff{omo} reads as
$$
\Delta v =\left( \sqrt{\Ll} -\sqrt{\lL}\right) \sqrt{- {\rm det}D^2v }\, .
$$
We will prove that the function
$$
\Phi(x)=\frac{x_n^{\Ll}}{|x|^{\Ll (n+1)-1}}
$$ 
actually is a  subsolution of \reff{omo}. Hence, we obtain  the upper bound
$$
\alpha \leq \Ll n -1
$$
and Theorem \ref{Lio} can be deduced as a  consequence of \reff{p*}. Note  that specific bounds for $\alpha$ are useful also when applying the extended comparison principles and the boundary singularity removability results given in \cite{ASS, Miller}, which require as assumptions growth conditions involving the exponent $\alpha$.

With  the subsolution $\Phi$ at hand, we can   bound from below not only the solution $\Phi_\alpha$, but all nonnegative supersolutions of \reff{omo}, and we  obtain a monotonicity property for supersolutions as in the classical three--circles Hadamard Theorem for superharmonic functions (see \cite{PW}). This will be performed in Section 2. 
Furthermore, we apply this monotonicity property in Section 3, where  we  provide  an alternative elementary proof of Theorem \ref{Lio} in the superlinear case $1\leq p \leq  \frac{\Ll n+1}{\Ll n-1}$.  Indeed, Theorem \ref{Lio} will be shown  to  follow  easily from our  nonlinear three--surfaces Hadamard theorem for $1\leq p < \frac{\Ll n+1}{\Ll n-1}$. In the limiting  case $p=\frac{\Ll n+1}{\Ll n-1}$ we will apply  a bootstrap argument: first, if $u$ satisfies \reff{mmeq}, then $u$ is   a supersolution of \reff{omo}, and then $u\geq c\, \Phi$ for some constant $c>0$ and in a suitable subdomain of $\Semi$. Therefore, by  \reff{mmeq} with $p=\frac{\Ll n+1}{\Ll n-1}$, we will have  that
$$
-\Mm (D^2 u)\geq c\, \left( \frac{x_n^{\Ll}}{|x|^{\Ll (n+1)+1}} \right)^{\frac{\Ll n+1}{\Ll n-1}}\, .
$$
Again, we will construct an explicit solution of the opposite inequality,   and the comparison principle will show that $u$ is too large to satisfy \reff{mmeq}.

\section{ Explicit subsolutions of $\Mm (D^2u)=0$ and  an Hadamard type theorem }

In this section we first of all construct an explicit   homogeneous  subsolution of the homogeneous equation $\Mm (D^2u)=0$ in the halfspace $\Semi$,  vanishing on $\partial \Semi \setminus \{0\}$. This will be then used to get information on solutions and supersolutions as well.
 
 We will make use of the following algebraic  result, whose proof is just  a straightforward computation.

\begin{lemma}\label{eigenvalues}
Let $v,\ w\in \Rn$ be unitary vectors and, given $a,  b, c, d\in \R$, let us consider the symmetric matrix
$$
A = a\,  v\otimes v + b \, w\otimes w + c \, (v\otimes w +w\otimes v) +d\, I_n\, ,
$$
where $v\otimes w$ denotes the $n\times n$ matrix whose $i, j$-entry is $v_i w_j$. Then, the eigenvalues of $A$ are:
\begin{itemize}

\item $d$, with multiplicity (at least) $n-2$ and eigenspace given by $<v, w>^{\bot}$;
\smallskip

\item $\displaystyle  d+\frac{a+b+2cv\cdot w\pm \sqrt{(a+b+2cv\cdot w)^2+4(1-(v\cdot w)^2)(c^2-ab)}}{2}$, which are simple (if different from $d$).
\end{itemize}
In particular, if either $c^2=ab$ or $(v\cdot w)^2=1$, then the eigenvalues are $d$, which has multiplicity $n-1$,  and $d+a+b+2cv\cdot w$, which is simple.
\end{lemma}
\smallskip

\begin{remark} {\rm Let us explicitely remark that the radicand appearing in the expression of the eigenvalues above is nonnegative, since 
$$
\begin{array}{l}
\dis (a+b+2cv\cdot w)^2+4(1-(v\cdot w)^2)(c^2-ab)\\[2ex]
\dis = (a-b)^2+4c v\cdot w (a+b) +4c^2 +4 (v\cdot w)^2ab\\[2ex]
\dis \geq \left( v\cdot w (a-b)\right)^2+4c v\cdot w (a+b) +4c^2 +4 (v\cdot w)^2ab\\[2ex]
=\left( v\cdot w (a+b)+2c\right)^2\geq 0
\end{array}
$$
}
\end{remark}
\smallskip

\begin{theorem}\label{solfsemi}
For any fixed $\Lambda \geq \lambda >0$, the function
\beq \label{fisemi}
\dis \Phi (x)=\frac{x_n^{\frac{\Lambda}{\lambda}}}{|x|^{\frac{\Lambda}{\lambda}(n+1)-1}}
\eeq
satisfies, in the classical sense,
\beq \label{m-semi}
 \Mm(D^2\Phi)\geq 0\qquad  in\ \ \Semi\, .
\eeq
\end{theorem}
\begin{proof}
Let us set $\ro =|x|$, and let us compute the hessian matrix for functions of the form
$$
\Phi (x)=\frac{x_n^\alpha}{\ro^\beta}
$$
for any  $\alpha ,\ \beta >0$. One has
$$\dis
D^2\Phi =\frac{x_n^\alpha}{\ro^{\beta+2}} \left[ \beta (\beta+2) \frac{x}{\ro}\otimes \frac{x}{\ro} + \alpha (\alpha -1) \left( \frac{\ro}{x_n}\right)^2 e_n\otimes e_n -\alpha \beta \frac{\ro}{x_n} \left( \frac{x}{\ro}\otimes e_n +e_n \otimes \frac{x}{\ro}\right) -\beta I_n\right]
$$
with $e_n=(0,1)\in \Rn$.

According to Lemma \ref{eigenvalues}, the eigenvalues $\mu_1, \ldots ,\mu_n$ of $D^2\Phi (x)$ are
\beq \label{mui}
\begin{array}{c}
\dis \mu_1 = \frac{x_n^\alpha}{\ro^{\beta+2}} \frac{\beta (\beta -2\alpha )+\alpha (\alpha -1)(\ro /x_n)^2 +\sqrt{\mathcal{D}}}{2}\\[2ex]
\dis \mu_2 = \frac{x_n^\alpha}{\ro^{\beta+2}} \frac{\beta (\beta -2\alpha )+\alpha (\alpha -1)(\ro /x_n)^2 -\sqrt{\mathcal{D}}}{2}\\[2ex]
\dis \mu_i = -\beta \frac{x_n^\alpha}{\ro^{\beta+2}} \qquad 3\leq i\leq n
\end{array}
\eeq
with
$$
\mathcal{D}=\left( \beta (\beta -2\alpha +2)+\alpha (\alpha -1)\left( \frac{\ro}{x_n}\right)^2\right)^2 +4 \alpha \beta (\beta -2\alpha +2)\left( \frac{\ro}{x_n}\right)^2 \left( 1- \left( \frac{x_n}{\ro}\right)^2\right)\, .
$$
We notice that
$$
\mathcal{D} =\left( \beta (\beta -2\alpha)+\alpha (\alpha -1)\left( \frac{\ro}{x_n}\right)^2\right)^2 +4 \beta (\beta -\alpha +1) \left( \beta -2\alpha +\alpha \left( \frac{\rho}{x_n}\right)^2\right)\, ;
$$
therefore, for $\beta \geq \alpha$, one has $\mu_1\geq 0$ and $\mu_i\leq 0$ for $2\leq i\leq n$. Hence
\beq\label{Mmfi}
\begin{array}{ll}
\dis \Mm (D^2\Phi) &  \dis \!\!\!\!  = \lambda\,  \mu_1 +\Lambda \sum_{i=2}^n \mu_i\\[3ex]
 & \dis  \!\!\!\! = \lambda \frac{x_n^\alpha}{\ro^{\beta+2}} \left[ \beta \left( \frac{1}{2}\left(  \Ll +1\right) (\beta -2\alpha )-\Ll (n-2)\right) \right. \\[3ex]
 & \dis \left. + \frac{\alpha}{2}(\alpha -1)\left( \Ll +1\right) \left( \frac{\ro}{x_n}\right)^2 -\frac{1}{2} \left( \Ll -1\right) \sqrt{\mathcal{D}}\right]\, .
 \end{array}
 \eeq
Furthermore,  the radicand $\mathcal{D}$  can be easily estimated as follows
 $$
 \begin{array}{ll}
 \dis \mathcal{D}  & \dis \leq \left( \beta (\beta -2\alpha +2)+\alpha (\alpha -1)\left( \frac{\ro}{x_n}\right)^2\right)^2 +4 \alpha \beta (\beta -2\alpha +2)\left( \frac{\ro}{x_n}\right)^2 \\[2ex]
 & \dis =  \beta^2 (\beta -2\alpha +2)^2+\alpha^2 (\alpha -1)^2\left( \frac{\ro}{x_n}\right)^4 +2 \alpha \beta (\alpha +1)(\beta -2\alpha +2)\left( \frac{\ro}{x_n}\right)^2\\[2ex]
 & \dis \leq \left( \beta (\beta -2\alpha +2)+\alpha (\alpha +1)\left( \frac{\ro}{x_n}\right)^2\right)^2 \, . 
\end{array}
$$
Inserting the above inequality into \reff{Mmfi} then yields
\beq \label{ineq1}
\Mm (D^2\Phi )\geq  \lambda \frac{x_n^\alpha}{\ro^{\beta+2}} \left[ \beta \left( \beta -2\alpha -\Ll (n-1)+1\right) +\alpha \left( \alpha -\Ll\right) \left( \frac{\ro}{x_n}\right)^2\right]\, .
\eeq
The choices $\alpha =\Ll\geq 1$ and $\beta =\Ll(n-1)+2 \alpha -1=\Ll (n+1)-1\geq \alpha$ then give \reff{m-semi}.
 
 \end{proof}

\begin{remark} {\rm   Let us point out that for $\Lambda =\lambda$ the  functions $\Phi$   coincides  with the harmonic function $\frac{x_n}{|x|^n}$, and equality holds in \reff{m-semi}. 
For $\Lambda >\lambda$, different choices for the exponents $\beta \geq \alpha >0$ are possible to make $\Phi (x)=\frac{x_n^\alpha}{|x|^\beta}$ a solution of \reff{m-semi}. Indeed, from the above proof   it follows that $\Phi$ satisfies \reff{m-semi} if and only if the following inequality holds true
\beq \label{dist}
\begin{array}{l}
\dis \beta \left( \left( \Ll +1\right) (\beta -2 \alpha )-2 \Ll (n-2) \right) t + \alpha (\alpha -1 )\left( \Ll +1\right)\\[2ex]
\dis \geq \left( \Ll -1\right) \sqrt{ \beta (\beta +2) (\beta -2 \alpha ) (\beta -2 \alpha +2) t^2 +\alpha ^2 (\alpha -1)^2 + 2 \alpha \beta (\alpha +1) (\beta -2 \alpha +2) t}\\
\hbox {for all }\ t\in [0,1]\quad \left( t=\left( \frac{x_n}{|x|}\right)^2 \right)\, .
\end{array}
\eeq
First,   we note that  testing \reff{dist} for $t=0$ yields $\alpha >1$. Then, we observe that \reff{dist} is satisfied also by $\beta= 2 \alpha$, $\alpha =2 \Ll (n-1)-1$. However, the smaller scaling exponent $\beta -\alpha =\Ll n-1$  selected in Theorem \ref{solfsemi} will produce better estimates.}
\end{remark}

\begin{remark}\label{super} {\rm As far as supersolutions for operator $\Mm$ are concerned, it is easy to prove that the function, already found in \cite{Miller},
$$
\hat{\Phi} (x) = \frac{x_n}{|x|^{\Ll (n-1)+1}}
$$
satisfies, in the classical sense,
$$
 \Mm (D^2\hat{\Phi})\leq 0\qquad  \hbox{in }\ \ \Semi\, .
$$
This can be  checked either directly, by using formulas  \reff{mui}, or by oserving that  $\Mm$ is superadditive and  $\hat{\Phi}$ is, up to a negative constant,  the derivative with respect to $x_n$ of  the well known radial solution for $\Mm$
\beq \label{radial}
\phi (x)=
\left\{
\begin{array}{ll}
\dis-\log |x| & \qquad
\mbox{if} \  \beta =2 \\[2ex]
\dis |x|^{2-\beta }& \qquad 
\mbox{if} \   \beta > 2  \\
\end{array}
\right.
\eeq
with $\beta =\Ll (n-1)+1$.
}\end{remark}

 The subsolution $\Phi$ given in Theorem \ref{solfsemi} can be used to estimate solutions and supersolutions by means of extended comparison principles of Phragm\'em--Lindel\"of type, such as the ones given in \cite{ ASS, Miller}. We present here another form of comparison principle, namely a nonlinear  three--surfaces version of the classical Hadamard three--circles theorem. Let us recall, see e.g. \cite{PW},  that this classical   result provides  a decay estimate at infinity for entire nonnegative superharmonic functions. More precisely, by comparing a nonnegative  function $u$ superharmonic in $\Rn$ with the fundamental solution, one has that  the function 
$m(r)=\inf_{B_r} u$ satisfies the concavity inequality
$$
\dis
m(r)\geq \left\{
\begin{array}{ll}
\dis \frac{m(r_2)\log (r_1/r)+m(r_1)\log (r/r_2)}{\log (r_1/r_2)}& \qquad
\mbox{if} \ n =2 \\[2ex]
\dis \frac{m(r_2)\left( r^{2-n }-r_1^{2-n}\right) +
m(r_1)\left( r_2^{2-n }-r^{2-n }\right)}
{\left( r_2^{2-n }-r_1^{2-n}\right)} & \qquad 
\mbox{if} \  n > 2\, . \\
\end{array}
\right.
$$
for every fixed $r_1>r_2>0$ and for all $r_2\leq r\leq r_1$. This immediately yields that $u$ is constant if $n=2$ (Liouville Theorem), and that $r\in (0,+\infty)\mapsto r^{n-2}m(r)$ is nondecreasing if $n\geq 3$.

The same argument can be used in the  fully nonlinear framework ,  see  \cite{CL}, where it has been proved that if $u$ is a bounded from below solution of $\Mm (D^2u)\leq 0$ in $\Rn$, then the infimum function $m(r)$ satisfies
\beq \label{hadam1}
m(r)\geq \left\{
\begin{array}{ll}
\dis \frac{m(r_2)\log (r_1/r)+m(r_1)\log (r/r_2)}{\log (r_1/r_2)}& \qquad
\mbox{if} \  \beta =2 \\[2ex]
\dis \frac{m(r_2)\left( r^{2-\beta }-r_1^{2-\beta}\right) +
m(r_1)\left( r_2^{2-\beta }-r^{2-\beta }\right)}
{\left( r_2^{2-\beta }-r_1^{2-\beta }\right)} & \qquad 
\mbox{if} \   \beta > 2\, . \\
\end{array}
\right.
\eeq
with $\beta =\frac{\Lambda}{\lambda}(n-1)+1$. This has been accomplished by comparing $u$ in annular domains with the new "fundamental solution", that is the radial solution of 
 $\Mm (D^2\phi)=0$ in $\Rn\setminus \{0\}$ given by \reff{radial}.

In order to obtain analogous results in $\Semi$, we have to consider suitable subdomains (suggested by the subsolution $\Phi$ of Theorem \ref{solfsemi}) where the comparison principle can be applied. 
 For $ x\in \Semi$, let us  define the positive function
 \beq \label{di}
 d=d(x)= \left( \frac{|x|}{x_n}\right)^k |x|\, ,
 \eeq
 with
$$
k=\frac{\Lambda -\lambda}{\Lambda \, n}\, ,
$$
 and  observe that $\Phi$ can be written as
 $$
 \Phi (x)= \frac{x_n}{d^{\Ll n}}\, .
 $$
Let us also
introduce, for every $r>0$, the sub--level sets
\beq \label{Br}
\mathcal{B}_r=\left\{ x\in \Semi \; \left| \, d (x)<r  \right. \right\} \, .
\eeq
We notice that, for $\Lambda =\lambda$, $d(x)$ reduces to  $|x|$ and the set $\mathcal{B}_r$ is nothing but the upper halfball $B^+_r= B_r\cap \Semi$.
 In the case
$\Lambda >\lambda$, $\mathcal{B}_r$ is an open subset of $B^+_r$, being $d(x)\geq |x|$. It is rotationally symmetric around the $x_n$--axis  and it satisfies
\beq \label{front}
\partial \mathcal{B}_r=\left\{ x\in \Semi  \; \left| \, d =r  \right.
\right\} \cup \left\{ (0,0)\right\}\, ,\quad \partial \mathcal{B}_r \cap \partial B^+_r=\left\{ (0,0),\ (0,r)\right\} \, .
\eeq
 
Let us consider now   a  lower semicontinuous function $u:\overline{\Semi}\to [0,+\infty ]$  satisfying in the viscosity sense
\beq \label{Mmsemi}
u\geq 0\,,\ \Mm (D^2u)\leq 0 \ \hbox{ in }\ \Semi\, .
\eeq
 By the strong maximum principle, if $u$ does not vanish identically then it is strictly positive in $\Semi$. Therefore, by   translating upward the domain if necessary, we can  assume that $u$ is   strictly positive on the closure  $\overline{\Semi}$. For positive $r$ let us define  the function
\beq \label{mu}
\dis \mu (r)=\inf_{x\in \mathcal{B}_r}\frac{u(x)}{x_n}\, .
 \eeq
Some immediate properties of  $\mu (r)$ are summarized in the following Lemma.

\begin{lemma}\label{mupro} Let $u$ be a positive lower semicontinuous function in $\overline \Semi$ satisfying \reff{Mmsemi},   and let
$\mu (r)$ be  defined by \reff{mu}. Then, for every $r>0$, there exists
a point $\hat x\in \partial \mathcal{B}_r\cap \Semi$ such that 
$$
\mu  (r)=\frac{u(\hat x)}{{\hat x}_n}\, .
$$
In particular,  $\mu  (r)$ is a positive and
decreasing function of $r\in (0 ,+\infty )$.
\end{lemma}

\begin{proof}
The function $\frac{u(x)}{x_n}$ is positive and lower semicontinuous
in $\overline \Semi$, so that the infimum $\mu  (r)$   actually
is a minimum on $\overline{\mathcal{B}_r}$, attained at some point belonging to
$\overline{\mathcal{B} _r}\cap \Semi$. Let us consider the function
$$
v_r(x)=u(x)-\mu  (r ) x_n\, ,
$$
which is nonnegative in $\overline {\mathcal{B}_r}$ and satisfies $\Mm (D^2v_r)\leq 0$ in $\mathcal{B}_r$. By the maximum principle the minimum of $v_r$ on $\overline {\mathcal{B}_r}$ is attained on $\partial  \mathcal{B}_r$. On the other hand, we have 
$\min_{\overline{\mathcal{B}_r}}v_r=0$ and $v_r=u>0$  for $x_n=0$,  so that from
\reff{front} the first part of the statement follows.
 
 Observing further that, for every $R>r>0$, one has $
\partial \mathcal{B}_R \cap \partial \mathcal{B}_r \cap \Semi
=\emptyset $, 
from the above it follows that $v_R(x)>0$ in $\overline{ \mathcal{B}_r}$, that is
$$
\frac{u(x)}{x_n} >\mu  (R)\qquad \forall \, x\in \overline{\mathcal{B}_r}\cap \Semi
 $$
and the claim is completely proved.

\end{proof}
 
We can now prove our nonlinear Hadamard type theorem.

\begin{theorem} \label{hadamsemith}Let $u: \overline \Semi \to [0,+\infty ]$ be a lower semicontinuous function satisfying  \reff{Mmsemi}. Then  the function $\mu (r)$ defined by
\reff{mu}   is a concave function of $r^{-\Ll n}$,
i.e. for every fixed $R>r>0$ and for all $r\leq \rho \leq R$ one has
\beq \label{hadsemi}
\dis \mu  (\rho )\geq \frac{\mu (r) \left( \rho^{-\Ll n} -R^{-\Ll n}\right) 
+\mu(R) \left( r^{-\Ll n} -\rho^{-\Ll n}\right)}
{r ^{-\Ll n} -R^{-\Ll n} }.
\eeq
Consequently, we have that
\beq \label{monosemi}
r\in (0,+\infty )\mapsto \mu  (r)\, r^{\Ll n} \quad is\
nondecreasing .
\eeq
\end{theorem}
\begin{proof}
We fix $R>r>0$ and we apply the comparison principle  in the domain
$\mathcal{B}_R \setminus \mathcal{B}_r$, where we consider the function
$$
\dis \Phi (x)= x_n \left( c_1 d  (x)^{-\frac{\Lambda}{\lambda}n}+c_2\right)   \, ,
$$
with constants $c_1\geq 0$ and  $c_2\in \R$ to be appropriately fixed. Notice that $\Phi$ has a
continuous extension in $\overline{\mathcal{B}_R \setminus \mathcal{B}_r}$ vanishing  at the origin. By Theorem \ref{solfsemi}, 
we have in particular
$$
\mathcal{M}^-_{\lambda ,\Lambda}(D^2\Phi )\geq 0\qquad \mbox{in }\ \mathcal{B}_R \setminus \mathcal{B}_r\, .
$$
 Let us now fix the constants $c_1\geq 0$ and $c_2\in \R$ in such a way that $\Phi \leq u$ on $\partial (\mathcal{B}_R \setminus \mathcal{B}_r)$. We impose
$$
\left\{
\begin{array}{l}
\dis c_1 r^{-\Ll n}+c_2 = \mu (r)\\[2ex]
\dis  c_1 R^{-\Ll n}+c_2 = \mu (R)
\end{array} \right. 
$$
which yields
$$
\left\{ \begin{array}{l}
\dis c_1 = \frac{\mu (r) -\mu (R)}{r^{-\Ll n}-R^{-\Ll n}}\geq 0\\[2ex]
\dis c_2= \frac{ \mu (R) r^{-\Ll n} - \mu(r) R^{-\Ll n}}{r^{-\Ll n}-R^{-\Ll n}}
 \end{array} \right.
$$
With this choice of $c_1$ and $c_2$ we can apply the comparison principle to   the subsolution $\Phi$ and to the supersolution $u$ in the domain 
$\mathcal{B}_R \setminus \mathcal{B}_r$,  which gives $\Phi \leq u$, that is 
$$
 \frac{u(x)}{x_n}\geq \frac{ \mu (r) \left( d(x)^{-\Ll n}-R^{-\Ll n}\right) + \mu (R) \left( r^{-\Ll n}- d(x)^{-\Ll n}\right)}{r^{-\Ll n}-R^{-\Ll n}}\, .
 $$
By Lemma \ref{mupro}, for every $r\leq \rho \leq R$ there exists a point $\hat x$ such that
$d (\hat x)=\rho$ and $\mu  (\rho )=\frac{u(\hat x)}{{\hat x}_n}$;
by applying the above inequality for $x=\hat x$, we then obtain  \reff{hadsemi}. 

By observing further that \reff{hadsemi} implies 
$$
\dis \mu (\rho )\geq \frac{\mu  (r)\left( 
\rho^{-\frac{\Lambda}{\lambda}n} -R^{-\frac{\Lambda}{\lambda}n}\right)} 
{r^{-\Ll n}-R^{-\Ll n}}
$$
and by letting $R\to +\infty$, we finally get the monotonicity property 
 \reff{monosemi}.
 
\end{proof}

As a consequence of Theorem \ref{hadamsemith} we  can obtain more specific bounds on the scaling exponent of the positive  singular homogeneous solution, which is already known to exist. By a  singular  homogeneous  function   we mean a positively  homogeneous function with  negative homogeneity exponent. We recall that, by the results of \cite{Miller} and their extensions in \cite{ASS}, it is known that there exists 
 a unique positive exponent $\alpha$,  and a unique $C^2$--function $\phi_\alpha :\left[ 0, \frac{\pi}{2}\right] \to \left[ 0, +\infty \right)$, with $\phi_\alpha (0)=0$, $\phi_\alpha (\theta)>0$ for $0< \theta \leq \frac{\pi}{2}$, such that  
\beq \label{fia}
\Phi_\alpha (x)= |x|^{-\alpha} \phi_\alpha \left( \arcsin \left( \frac{x_n}{|x|}\right) \right)
\eeq
 is the unique (up to normalization)   singular homogeneous  and   continuous in $\overline{\Semi}\setminus \{0\}$ solution of
$$
\Phi_\alpha >0\, ,\ \Mm (D^2 \Phi_\alpha )=0 \ \hbox{ in } \Semi\, ,\quad \Phi_\alpha =0 \ \hbox{ on } \partial \Semi \setminus \{ 0\}\, .
$$
Moreover, as observed in \cite{Miller}, a comparison argument  applied to $\Phi_\alpha$ and the supersolution $\hat{\Phi}$ given in Remark \ref{super}, yields
\beq\label{lbound}
\alpha \geq \Ll (n-1)\, .
\eeq
On the other hand, by Theorem \ref{hadamsemith} we immediately obtain the following upper bound.

\begin{corollary}\label{alfa} The scaling exponent of the solution $\Phi_\alpha$ in \reff{fia} satisfies
$$
\alpha\leq \Ll n -1\, .
$$
\end{corollary}

\begin{proof} It is enough to observe that the function $\phi_\alpha$ can be normalized in order  to satisfy
$$
\phi_\alpha (\theta )\leq \sin \theta \qquad \hbox{ for all } \ \theta \in \left[ 0, \frac{\pi}{2} \right]\, .
$$
Therefore, the infimum function $\mu (r)$  for $\Phi_\alpha$ satisfies $\mu (r)\leq \frac{1}{r^{\alpha +1}}$, and  \reff{monosemi} is violated for $\alpha > \Ll n-1$.

\end{proof}

\begin{remark}{\rm  
We cannot prove that the exponent appearing in the growth condition \reff{monosemi} is sharp, since it is derived by a comparison argument with a subsolution, not a solution. In order to obtain the optimal condition, we can repeat the proof of Theorem \ref{hadamsemith} with the subsolution $\Phi$ replaced by  the solution $\Phi_\alpha$. In this case, we consider, for positive $r$, the function
$$
m (r) =\inf_{B_r^+} \frac{u(x)}{x_n}\, .
$$ 
We further observe that, by Hopf's Lemma, the function $\phi_\alpha$ satisfies  $\phi_\alpha'(0)>0$, so that, up to a normalization, one has
$$
c\, \sin \theta \leq \phi_\alpha (\theta )\leq \sin \theta \qquad \hbox{for all }\ \theta \in \left[ 0, \frac{\pi}{2}\right] \, ,
$$
with $c =\inf_{ \left( 0, \frac{\pi}{2}\right) } \frac{\phi_\alpha (\theta)}{\sin \theta} >0$ depending only on $\Lambda ,\ \lambda$ and $n$.  As in the proof of Theorem \ref{hadamsemith}, the comparison principle applied in the upper annular domain $B_R^+\setminus B_r^+$ then yields that $m(r)$ satisfies
$$
\dis m  (\rho )\geq \frac{m (r) \left( c\,  \rho^{-(\alpha +1)} -R^{- (\alpha +1)}\right) 
+ m (R) \left( r^{-(\alpha +1)} -c\, \rho^{-(\alpha +1)}\right)}
{r ^{-(\alpha +1)} -R^{-(\alpha +1)} }.
$$
for any fixed $R>r>0$ and all $r\leq \rho \leq R$. Hence,
$$
\rho^{\alpha +1} m(\rho )\geq c\, r^{\alpha +1} m(r) \qquad \hbox{ for all }\ \rho \geq r\, .
$$
}\end{remark}

\begin{remark} \label{emmepiu} {\rm 
All the statements we have given in this section for operator $\Mm$ correspond to analogous results  for operator $\Mp$. In particular, the function
$$
\hat{\Psi} (x) = \frac{x_n}{|x|^{\lL (n-1)+1}}\, ,
$$
satisfies, in the classical sense,
$$
 \Mp (D^2\hat{\Psi})\geq 0\qquad  \hbox{in }\ \ \Semi\, .
$$
Note that $\hat \Psi$
 is, up to a negative constant, the partial derivative with respect to $x_n$ of the radial solution for $\Mp$
$$
\psi (x)=
\left\{
\begin{array}{ll}
\dis -|x|^{2-\beta} & \qquad \mbox{if}\ \beta <2\\[2ex]
\dis-\log |x| & \qquad
\mbox{if} \  \beta =2 \\[2ex]
\dis |x|^{2-\beta }& \qquad 
\mbox{if} \   \beta > 2 \\
\end{array}
\right.
$$
with $\beta =\lL (n-1)+1$.
Therefore, the same proof of Theorem \ref{hadamsemith}, which is in this case even simpler, yields that, if $u$ is a supersolution for $\Mp$, then the function
$$
m (r) =\inf_{B_r^+} \frac{u(x)}{x_n}
$$
    is a concave function of $r^{-\left( \lL (n-1)+1\right)}$,
i.e. for every fixed $R>r>0$ and for all $r\leq \rho \leq R$ one has 
$$
\dis m  (\rho )\geq \frac{m (r) \left( \rho^{-\left( \lL (n-1)+1\right)} -R^{-\left( \lL (n-1)+1\right)}\right) 
+m (R) \left( r^{-\left( \lL (n-1)+1\right)} -\rho^{-\left( \lL (n-1)+1\right)}\right)}
{r ^{-\left( \lL (n-1)+1\right)} -R^{-\left( \lL (n-1)+1\right)} }.
$$
Hence,   
$$
r\in (0,+\infty )\mapsto m  (r)\, r^{ \lL (n-1)+1} \quad \hbox{ is
nondecreasing .}
$$
As far as  supersolutions are concerned, 
 the same proof of Theorem \ref{solfsemi} carried out for operator $\Mp$ shows  that, under the assumption $\lL n\geq 1$, the function
$$
\dis \Psi (x)=\frac{x_n^{\frac{\lambda}{\Lambda}}}{|x|^{\frac{\lambda}{\Lambda}(n+1)-1}}
$$
satisfies, in the classical sense,
$$
 \Mp (D^2\Psi)\leq 0\qquad  \hbox{in }\ \ \Semi\, .
$$
It then follows that the  positive singular homogeneous solution $\Psi_\alpha$ for operator $\Mp$ has a positive scaling  exponent $\alpha =\alpha (\Mp)$ satisfying
$$
\lL n-1 \leq \alpha \leq \lL (n-1)\, .
$$
This improves the lower bound $\alpha \geq \lL (n-1)-1$ proved in \cite{Miller} by comparing $\Psi_\alpha$ with the radial (super)solution $\psi$ in the case $\lL (n-1)>1$.
}\end{remark}

\section{ Explicit subsolutions and a  Liouville type theorem}

In this section we give an elementary proof, purely based  on the comparison principle,  of the following Liouville type theorem for inequalities with   superlinear zero order terms.

\begin{theorem} \label{Lio2} Let $n\geq 2$ and $1\leq p\leq \frac{\Ll n+1}{\Ll n-1}$. Then,  $u\equiv 0$ is the only nonnegative viscosity solution of inequality 
\beq \label{mmp}
\Mm (D^2u)+u^p\leq 0\qquad \hbox{ in }\ \Semi\, .
\eeq
\end{theorem}

To prove the above result in the limiting  case $p=\frac{\Ll n+1}{\Ll n-1}$, we will compare the supersolution $u$ with an explicit subsolution of the equation
$$
-\Mm (D^2v)= \left( \frac{x_n}{d^{\Ll n}}\right)^{\frac{\Ll n+1}{\Ll n-1}}
$$
where $d=d(x)$ is as in \reff{di}. Such a subsolution is constructed in the following preliminary result.

\begin{lemma} \label{Gammasottosol} There exist positive constants $a,\ b>0$ and $d_0\geq 1$, depending only on $\lambda,\ \Lambda$ and $n$, such that the function
\beq \label{Gamma}
\Gamma (x) =\frac{x_n}{d^{\Ll n}} \left( a \ln d + b \left( \frac{x_n}{|x|}\right)^2\right)
\eeq
satisfies, in the classical sense,
\beq \label{eqGamma}
-\Mm (D^2\Gamma )\leq  \left( \frac{x_n}{d^{\Ll n}}\right)^{\frac{\Ll n+1}{\Ll n-1}}\qquad \hbox{ in } \Semi \setminus \overline{ \mathcal{B}_{d_0}}\, .
\eeq
\end{lemma}

\begin{proof} Let us consider the two functions
$$
\Gamma_1(x)=\frac{x_n}{d^{\Ll n}} \ln d\
$$ 
and
$$
 \Gamma_2(x)=\frac{x_n}{d^{\Ll n}} \left( \frac{x_n}{\rho}\right)^2= \frac{x_n^{\Ll +2}}{\rho^{\Ll (n+1)+1}}\, ,
$$
with $\rho=|x|$. 
If $a, b >0$ and $\Gamma$ is given by \reff{Gamma}, then, being  $\Mm$  superadditive and positively homogeneous, we have that
\beq \label{semiest1}
-\Mm (D^2\Gamma) \leq  -a\,  \Mm (D^2\Gamma_1) - b \, \Mm (D^2\Gamma_2) \,.
\eeq
Therefore, in order to prove \reff{eqGamma}, we estimate separately the two terms appearing in the right hand side of  \reff{semiest1}.

As far as $\Gamma_1$ is concerned, definition \reff{di} of $d$ and a direct computation show that
$$
\begin{array}{ll}
\dis  D^2\Gamma_1 (x) = & \dis \!\!   (k+1) \frac{x_n}{d^{\Ll n+2}}\left( \frac{\rho}{x_n}\right)^{2k}  \left\{ \left[ \Ll n \left(\Ll (n+1)+1\right) \ln d -2 \Ll (n+1)\right] \frac{x}{\rho} \otimes \frac{x}{\rho} \right.\\[2ex]
& \dis  + \frac{k}{k+1} \left(\frac{\rho}{x_n}\right)^2\left( \left(\Ll\right)^2n \ln d -2\Ll +1\right) e_n\otimes e_n \\[2ex]
& \dis \left. -\frac{\rho}{x_n} \left(  \left(\Ll\right)^2n \ln d-2\Ll +1\right) \left( \frac{x}{\rho}\otimes e_n + e_n \otimes \frac{x}{\rho}\right) 
 - \left( \Ll n \ln d -1\right) I_n \right\}
\end{array}
$$
with $k= \frac{\Lambda -\lambda}{\Lambda n}$. According to Lemma \ref{eigenvalues}, the eigenvalues $\mu_1 ,\ldots , \mu_n$ of $D^2 \Gamma_1 (x)$ are
$$
\begin{array}{c}
\begin{array}{ll}
\dis \mu_1=   \frac{k+1}{2}\frac{x_n}{d^{\Ll n+2}} \left( \frac{\rho}{x_n}\right)^{2k} & \dis \!\!\!  \left[ \Ll n \left( \Ll (n-1)-1\right) \ln d-2\Ll  (n-1)\right. \\[2ex]
 & \dis \left. +\frac{k}{k+1}\left(\frac{\rho}{x_n}\right)^2 \left( \left( \Ll\right)^2n \ln d -2\Ll +1\right)+ \sqrt{\mathcal{D}}\right] \end{array}\\[6ex]
\begin{array}{ll}
\dis \mu_2=   \frac{k+1}{2}\frac{x_n}{d^{\Ll n+2}} \left( \frac{\rho}{x_n}\right)^{2k} & \dis \!\!\!  \left[ \Ll n \left( \Ll (n-1)-1\right) \ln d-2\Ll  (n-1)\right. \\[2ex]
 & \dis \left. +\frac{k}{k+1}\left(\frac{\rho}{x_n}\right)^2 \left( \left( \Ll\right)^2n \ln d -2\Ll +1\right)- \sqrt{\mathcal{D}}\right] \end{array}\\[6ex]
\dis \mu_i = -  (k+1) \frac{x_n}{d^{\Ll n+2}}\left( \frac{\rho}{x_n}\right)^{2k} \left( \Ll n \ln d -1\right)\, ,\quad 3\leq i\leq n\, ,
\end{array}
$$
where
\beq \label{D1}
\begin{array}{ll}
\dis \mathcal{D} =  & \!\!  \dis\left[ \left( \Ll (n-1)+1\right) \left( \Ll n \ln d -2\right) +\frac{k}{k+1} \left( \frac{\rho}{x_n}\right)^2 \left( \left( \Ll\right)^2 n\ln d -2 \Ll +1\right) \right]^2 \\[4ex]
& \dis
+ 4  \, \frac{\Ll (n-1)+1}{\Ll (n+1)-1} \left ( 1- \left( \frac{x_n}{\rho}\right)^2\right) \left( \frac{\rho}{x_n}\right)^2 \left( \left( \Ll\right)^2 n \ln d-2 \Ll +1\right) \left( \Ll n \ln d -1\right)
\end{array}
\eeq
For $d\geq d_0$,  with $d_0$ depending only on $\Lambda ,\ \lambda$ and $n$, it is easy to see that $\mu_1\geq 0$ and $\mu_i\leq 0$ for $2\leq i\leq n$. Therefore, one has
\beq \label{mmg1}
\begin{array}{ll}
\dis \Mm (D^2\Gamma_1 (x)) & \!\! \dis = \lambda \, \mu_1 +\Lambda\, \sum_{i=2}^n \mu_i\\[2ex]
 & \!\! \dis = - \frac{\lambda}{2} (k+1) \frac{x_n}{d^{\Ll n+2}} \left( \frac{\rho}{x_n}\right)^{2k} \left[  \Ll  \left( \Ll (n-1)+1\right) \left( - n \left( \Ll -1\right) \ln d +2\right) \right.\\[3ex]
  &  \dis \left. - \frac{k}{k+1} \left( \Ll +1\right) \left( \frac{\rho}{x_n}\right)^2 \left( \left(\Ll\right)^2 n \ln d -2\Ll +1\right) +\left( \Ll -1\right) \sqrt{\mathcal{D}} \right]
  \end{array}
  \eeq
  Moreover, from \reff{D1} it follows that
  $$
\begin{array}{rl}
\dis \mathcal{D} \leq   & \!\!\!\!  \dis\left[ \left( \Ll (n-1)+1\right) \left( \Ll n \ln d -2\right) +\frac{k}{k+1} \left( \frac{\rho}{x_n}\right)^2 \left( \left( \Ll\right)^2 n\ln d -2 \Ll +1\right) \right]^2 \\[4ex]
& \dis
+ 4  \, \frac{\Ll (n-1)+1}{\Ll (n+1)-1}  \left( \frac{\rho}{x_n}\right)^2 \left( \left( \Ll\right)^2 n \ln d-2 \Ll +1\right) \left( \Ll n \ln d -1\right)\\[4ex]
\leq & \!\!\!\!  \dis  \left[ \left( \Ll (n-1)+1\right) \left( \Ll n \ln d -\frac{2\Ll}{\Ll +1}\right) +\frac{\Ll +1}{\Ll (n+1)-1} \left( \frac{\rho}{x_n}\right)^2 \left( \left( \Ll\right)^2 n \ln d-2 \Ll +1\right) \right]^2
\end{array}
$$
The above estimate plugged into \reff{mmg1} gives
\beq \label{semiest2}
\Mm  (D^2\Gamma_1 (x)) \geq - \frac{ 2 \lambda}{\Ll +1} \frac{ \left( \Ll (n-1)+1\right) \left( \Ll (n+1)-1\right)}{n} \frac{x_n}{d^{\Ll n+2}} \left( \frac{\rho}{x_n}\right)^{2k}
=- \frac{c_1}{d^{\Ll n+1}}\left( \frac{x_n}{\rho}\right)^{1-k}
\eeq
with $c_1=\frac{ 2 \lambda}{\Ll +1} \frac{ \left( \Ll (n-1)+1\right) \left( \Ll (n+1)-1\right)}{n}$.

Let us now turn to estimate $\Mm  (D^2\Gamma_2 (x))$. By applying inequality \reff{ineq1}  with $\alpha =\Ll +2$ and $\beta= \Ll (n+1)+1 > \alpha$, we obtain
\beq \label{semiest3}
\begin{array}{rl}
\dis \Mm  (D^2\Gamma_2 (x)) \geq  & \dis -  2 \lambda    \frac{x_n^{\Ll +2}}{\rho^{\Ll (n+1)+3}} \left(  \Ll (n+1)+1 -\left( \Ll +2\right) \left( \frac{\rho}{x_n}\right)^2\right)\\[2ex]
= & \dis
- \frac{1}{d^{\Ll n+1}} \left( \frac{x_n}{\rho}\right)^{1-k}  \left( c_2 \left( \frac{x_n}{\rho}\right)^2 - c_3\right) 
\end{array}
\eeq
where $c_2=2 \left( \Lambda (n+1) +\lambda \right)$ and $c_3= 2\left( \Lambda +2\lambda\right)$.

Inequalities \reff{semiest2} and \reff{semiest3} combined with \reff{semiest1} with $a=\frac{c_3}{c_1 c_2}$ and $b=\frac{1}{c_2}$ then imply
$$
-\Mm (D^2\Gamma) \leq \frac{1}{d^{\Ll n+1}} \left( \frac{x_n}{\rho}\right)^{3-k}
$$
We finally observe that , since $\Ll (n-1)\geq 1$, one has $3-k\geq (1+k) \frac{\Ll n+1}{\Ll n-1}$. Hence
$$
 -\Mm (D^2\Gamma) \leq \frac{1}{d^{\Ll n+1}} \left( \frac{x_n}{\rho}\right)^{ (1+k) \frac{\Ll n+1}{\Ll n-1}}= \left( \frac{x_n}{d^{\Ll n}}\right)^{\frac{\Ll n+1}{\Ll n-1}}
 $$
\end{proof}

\begin{remark} {\rm Let us observe that in the linear planar case, that is for $\Lambda =\lambda$ and $n=2$ inequality \reff{eqGamma} becomes equality.}
\end{remark}

\noi {\sl Proof of Theorem \ref{Lio2}.} 
For a contradiction, let us assume that there exists a non trivial solution $u$ of \reff{mmp}. As in the previous section, by using the strong maximum principle and by translating upward the domain if necessary, we can assume without loss that $u$ is strictly positive in $\overline{\Semi}$.

Let us re--scale inequality \reff{mmp}, that is, for every $r>0$ let us set
$$u_r(x) =u(r x)\, .
$$
Then, $u_r$ satisfies
\beq \label{mmpr}
u_r>0 \ \hbox{ in } \overline{\Semi}\, ,\qquad  \Mm (D^2 u_r) + r^2 u_r^p\leq 0 \ \hbox{ in } \Semi \, .
\eeq
We now test inequality \reff{mmpr} with a suitable cut--off function, chosen    constant on the ball $B_{1/2}((0,1))$ centered at $(0,1)$ and having radius $1/2$, and negative outside $B_{3/4}((0,1))$. Precisely, let us select  a smooth, concave, non increasing function $\zeta :[0,+\infty )\to \R$ satisfying
$$
\zeta(t) =\left\{
\begin{array}{ll}
\dis 1 & \dis \hbox{ for }\ 0\leq t \leq 1/2\\
>0  & \dis \hbox{ for }\ 1/2 < t  < 3/4 \\
\leq 0 & \dis \hbox{ for }\  t\geq 3/4
\end{array}
\right.
$$
and let us consider the radial function
$$
z (x) = \left( \inf_{B_{1/2}((0,1))} u_r \right) \zeta (|x- (0,1)|)\, .
$$
Note that $u_r \geq z$ in $\overline{B_{1/2}((0,1))}$, $u_r = z$ at some point in  $\partial B_{1/2}((0,1))$ and $u_r >  z$ outside $B_{3/4}((0,1))$. Therefore, the infimum of $u-z$ is non positive and it is achieved at some point $x^*\in B_{3/4}((0,1)) \setminus B_{1/2}((0,1))$. By definition of viscosity solution of inequality \reff{mmpr}, it then follows
\beq \label{test1}
u_r(x^*)^p \leq  \frac{C}{r^2} \inf_{B_{1/2}((0,1))} u_r\, ,
\eeq
where 
$$
\dis C=\sup_{B_{3/4}((0,1))} \left( -\Mm (D^2 \zeta )\right) =\sup_{B_{3/4}((0,1))} \left( -\Lambda\,  \Delta  \zeta \right)= - \Lambda \, \inf_{\left[ 1/2, 3/4\right]} \left( \zeta ''(t) +(n-1) \frac{\zeta'(t)}{t}\right)
$$
 is a positive constant depending only on $\Lambda$ and $n$.
Here and throughout in the sequel,  we will use $c$ and $C$ to denote positive constants,   which may change from line to line, not  depending on $r$. 

By the nonlinear Hadamard three--spheres theorem for positive supersolutions of the equation $\Mm (D^2u)=0$ (see \reff{hadam1} and Theorem 3.1 in \cite{CL}), we also have
$$
\inf_{B_{1/2}((0,1))} u_r \leq C\, \inf_{B_{3/4}((0,1))} u_r \, ,
$$
and from \reff{test1} it  then follows
$$
\left(  \inf_{B_{3/4}((0,1))} u_r \right)^p\leq u_r(x^*)^p\leq \frac{C}{r^2}  \inf_{B_{3/4}((0,1))} u_r \, .
$$
Now,  the contradiction is evident if $p=1$. For $p>1$, we re--scale back from  $u_r$ to $u$ and we further observe that
$$
 \inf_{B_{3/4}((0,1))} u_r  =  \inf_{B_{\frac{3}{4}r}((0,r))} u \geq \frac{r}{4} \inf_{B_{\frac{3}{4}r}((0,r))} \frac{u}{x_n} \geq \frac{r}{4}\,  \inf_{\mathcal{B}_{2r}} \frac{u}{x_n} =\frac{r}{4} \,  \mu (2r)\, ,
 $$
 where $\mu$ is defined in \reff{mu}. Hence, we obtain
\beq \label{upb2}
\mu (r)\leq \frac{C}{r^{\frac{p+1}{p-1}}}\, .
\eeq
If $ \frac{p+1}{p-1}> \Ll n$, that is if $p<\frac{\Ll n+1}{\Ll n-1}$, then \reff{upb2} contradicts the monotonicity property \reff{monosemi}. Thus, only the case $p=\frac{\Ll n+1}{\Ll n-1}$ remains to be considered.
In this case, \reff{upb2} gives the upper bound
\beq \label{upb3}
r^{\Ll n} \mu (r)\leq C \quad \hbox{ for all }\ r>0\, .
\eeq
On the other hand, by \reff{monosemi} we also have
$$
r^{\Ll n} \mu (r)\geq d_0^{\Ll n} \mu (d_0)= c >0 \quad \hbox{ for all }\ r \geq d_0\, ,
$$
which implies
$$
u(x) \geq c \frac{x_n}{ d(x)^{\Ll n}} \quad \hbox{ for }\ x\in \Semi \setminus \mathcal{B}_{d_0}\, ,
$$
where $d_0>0$ is given by Lemma \ref{Gammasottosol}. By inequality \reff{mmp} with $p=\frac{\Ll n+1}{\Ll n-1}$ it then follows that $u$ satisfies
\beq \label{eq2}
-\Mm (D^2u) \geq c \left( \frac{x_n}{ d(x)^{\Ll n}} \right)^{\frac{\Ll n+1}{\Ll n-1}} \quad \hbox{ in }\   \Semi \setminus \mathcal{B}_{d_0}\, .
\eeq
By Lemma \ref{Gammasottosol}, the opposite inequality is satisfied by  $\gamma\, \Gamma (x)$, where $\Gamma$ is  given by \reff{Gamma} and $0<\gamma \leq c$. If $\gamma$ is further assumed to satisfy $\gamma \leq \mu (d_0) \frac{d_0^{\Ll n}}{a \ln d_o +b}$, then we have
$$
\gamma\, \Gamma (x) \leq u(x) \quad \hbox{ on } \partial \mathcal{B}_{d_0}\, .
$$
Moreover, for any fixed $\epsilon >0$,  let $R>0$ be large enough so that
$$
\gamma\, \Gamma (x) \leq \epsilon \quad \hbox{ for } x\in \Semi \setminus \mathcal{B}_{R}\, .
$$
The comparison principle applied to $\gamma\, \Gamma$ and $u+ \epsilon$ in $\mathcal{B}_{R} \setminus \overline{\mathcal{B}_{d_0}}$ then gives $\gamma\, \Gamma (x)\leq u(x) +\epsilon$ in $\overline{\mathcal{B}_{R}} \setminus  \mathcal{B}_{d_0}$ for all $R$ sufficiently large. If we let first $R\to \infty$ and then $\epsilon \to 0$,   we obtain
$$
u(x)\geq \gamma\, \Gamma (x) \quad \hbox{ for } x\in \Semi \setminus \mathcal{B}_{d_0}\, ,
$$
which yields, by \reff{Gamma},
$$u(x) \geq c\, \frac{x_n}{d(x)^{\Ll n}}\ln d(x)\quad \hbox{ for } x\in \Semi \setminus \mathcal{B}_{d_0}\, .
$$
By Lemma \ref{mupro}, this implies that $r^{\Ll n} \mu(r) \geq c\, \ln r$ for all $r\geq d_0$, and this contradicts  \reff{upb3}.

\hfill $\Box$

\begin{remark} {\rm Theorem \ref{Lio2} of course still holds if $u$ is a supersolution in the exterior domain $\Semi \setminus \overline{B_r}$ for any $r>0$. In fact, in this case inequality \reff{mmeq} is satisfied in the translated halfspace $\{ x\in \R^n \; :\;  x_n>r \}$.}
\end{remark}

\begin{remark} {\rm By the characterization \reff{p*} of the critical exponents $p_*$ and $p^*$, and by the lower bound \reff{lbound}, we know that the Liouville property does not hold for inequality \reff{mmp} if either $p<-1$ or $p>\frac{\Ll (n-1)+2}{\Ll(n-1)}$. This can be checked also directly, by finding some explicit supersolution $u$. Indeed, it is immediate to verify that, if $p<-1$, then $u(x)= x_n^\delta$, with $1>\delta > \frac{2}{1-p}$, is a supersolution in the halfspace $\{ x_n \geq \left( \Lambda \, \delta (1-\delta)\right)^{\frac{1}{\delta (p-1)+2}}\}$.
On the other hand,  for $p>\frac{\Ll (n-1)+2}{\Ll(n-1)}$ we can consider the function
$$
u(x)=\frac{x_n}{|x|^\beta}
$$
with $\Ll (n-1)+1> \beta >\frac{p+1}{p-1}$, which satisfies, by  formula \reff{Mmfi} with $\alpha =1$, 
$$
\begin{array}{rl}
\dis \Mm (D^2 u) = & \dis  \frac{\beta \lambda }{2} \frac{x_n}{|x|^{\beta +2}} \left[ \left( \Ll +1\right) \beta -2 \Ll (n-1) -2  - \left( \Ll -1\right) \sqrt{ \beta^2 +4 \left( \left( \frac{|x|}{x_n}\right)^2 -1\right)}\right]\\[4ex]
 \leq & \dis - \beta \lambda \frac{x_n}{|x|^{\beta +2}} \left( \Ll (n-1)+1 -\beta \right)\leq - u^p
 \end{array}
$$
for $x\in \Semi \setminus \overline{B_r}$, with $r=\left( \lambda \beta \left( \Ll (n-1)+1 -\beta \right)\right)^{-\frac{1}{\beta(p-1)-p-1}}$.

Let us observe that for $\Lambda >\lambda$, one has $\frac{\Ll n+1}{\Ll n-1}< \frac{\Ll (n-1)+2}{\Ll(n-1)}$, so that  in this case the existence or non existence of solutions for \reff{mmp} is somehow indeterminate  for $\frac{\Ll n+1}{\Ll n-1}<p \leq  \frac{\Ll (n-1)+2}{\Ll(n-1)}$.

Analogously, by \reff{p*} and Remark \ref{emmepiu}, it follows that the inequality
$$
\Mp  (D^2u)+u^p\leq 0\qquad \hbox{ in }\ \Semi
$$
does not have any positive solution for $-1 \leq p \leq \frac{\lL (n-1)+2}{\lL (n-1)}$. Positive solutions do exist if either $p<-1$ or  $p> \frac{\lL n +1}{\lL n-1}$, provided that $\lL n>1$. Note that if $\lL n\leq 1$, no upper bound for $p^*(\Mp)$ is given.
}
\end{remark}

\begin{remark} {\rm By applying the proof of Theorem 5.1 given in \cite{AS}, with the functions $\Psi^+$ and $\Psi^-$ there replaced respectively by $\Phi_\alpha$ given in \reff{fia} and by $x_n$, and by using Corollary \ref{alfa}, a more general result than Theorem \ref{Lio2} can be obtained. Precisely, it can be proved the following statement:
\\
\noi \emph{  let $r_0\geq 0$, $\gamma >-2$ and $f:(0,+\infty )\to (0,+\infty )$ be continuous with
$$
\liminf_{t \to 0} t^{-\frac{\Ll n+1+\gamma}{\Ll n-1}} f(t) >0 \quad \hbox{ and }\quad \liminf_{t \to +\infty} t^{1+\gamma} f(t) >0\, .
$$
Then,  there does not exist any positive solution of
$$
\Mm (D^2u) +|x|^\gamma f(u) \leq 0 \qquad \hbox{in }\ \Semi \setminus \overline{B_{r_0}}\, .
$$
}Due to the precence of the general nonlinearity $f$,  the proof relies on Alexandrov--Bakelmann--Pucci estimate  and a weakened  form of weak Harnack inequality. }\end{remark}

\end{document}